\documentclass[11pt]{amsart}

\newcommand{\R}{\mathbb{R}}
\newcommand{\C}{\mathbb{C}}
\newcommand{\N}{\mathbb{N}}
\newcommand{\Z}{\mathbb{Z}}

\newtheorem{thm}{Theorem}

\newtheorem{prop}[thm]{Proposition}
\newtheorem{lem}[thm]{Lemma}
\newtheorem{cor}[thm]{Corollary}
\newtheorem{rem}[thm]{Remark}

\begin{document}

\title{Boundaries of Levi-flat hypersurfaces: special hyperbolic points}

\author{Pierre Dolbeault}

\address{Institut de Math\'ematiques de Jussieu, UPMC, 4, place Jussieu 75005 Paris}

\email{pierre.dolbeault@upmc.fr}

\date{\today}

\begin{abstract}
Let $S\subset \C^n$, $n\geq 3$ be a compact connected 2-codimensional submanifold having the following property: there exists a Levi-flat hypersurface whose boundary is $S$, possibly as a current. Our goal is to get examples of such $S$ containing at least one special 1-hyperbolic point: sphere with two horns; elementary models and their gluing. The particular cases of graphs are also described.
\end{abstract}

\maketitle

\thispagestyle{empty}


\section{Introduction}\label{Sec:1}

Let $S\subset \C^n$, be a compact connected 2-codimensional submanifold having the following property: there exists a Levi-flat hypersurface $M\subset \C^n\setminus S$ such that $dM=S$ (i.e. whose boundary is $S$, possibly as a current). The case $n=2$ has been intensively studied since the beginning of the eighties, in particular by Bedford, Gaveau, Klingenberg; Shcherbina, Chirka, G. Tomassini, Slodkowski, Gromov, Eliashberg; it needs global conditions: $S$ has to be contained in the boundary of a srictly pseudoconvex domain.

We consider the case $n\geq 3$; results on this case has been obtained since 2005 by Dolbeault, Tomassini and Zaitsev, local necessary conditions recalled in section 2 have to be satisfied by $S$, the singular CR points on $S$ are supposed to be elliptic and the solution $M$ is obtained in the sense of currents \cite{DTZ05,DTZ10}. More recently a regular solution $M$ has been obtained when $S$ satisfies a supplementary global condition as in the case $n=2$ \cite{DTZ09}, the singular CR points on $S$ still supposed to be elliptic.

The problem we are interested in is to get examples of such $S$ containing at least one special 1-hyperbolic point (section \ref{Sec:2.4}). The CR-orbits near a special 1-hyperbolic point are large and, assuming them compact, a careful examination has to be done (sections \ref{Sec:2.6}, \ref{Sec:2.7}). As a topological preliminary, we need a generalization of a theorem of Bishop on the difference of the numbers of special elliptic and 1-hyperbolic points (section \ref{Sec:2.8}); this result is a particular case of a theorem of Hon-Fei Lai \cite{L72}.

The first considered example is the sphere with two horns which has one special 1-hyperbolic point and three special elliptic points (section 3.4). Then we consider elementary models and their gluing to obtain more complicated examples (section 3.5). Results have been announced in \cite{D08}, and in more precise way in \cite{D09}; the first aim of this paper is to give complete proofs. Finally, we recall in detail and extend the results of \cite{DTZ09} on regularity of the solution when $S$ is a graph satisfying a supplementary global condition, as in the case $n=2$, to the case of existence of special 1-hyperbolic points, and to gluing of elemetary smooth models (section 4).

\section{Preliminaries: local and global properties of the boundary}\label{Sec:2}

\subsection{Definitions}\label{Sec:2.1}

A smooth, connected, CR submanifold $M\subset \C^n$ is called {\it minimal}
at a point
$p$ if there does not exist a submanifold
$N$ of $M$ of lower dimension through $p$ such that $HN = HM|_{N}$.
By a theorem of Sussman, all possible submanifolds $N$ such that $HN = HM|_{N}$ contain, as
germs at $p$, one of the
minimal possible dimension, defining a so called CR {\it orbit} of $p$ in $M$
whose germ at $p$ is uniquely determined.

\smallskip

Let $S$ be a smooth compact connected oriented submanifold of dimension $2n-2$. $S$ is said to be a {\it locally flat boundary} at a
point
$p$ if it locally bounds a Levi-flat hypersurface near $p$. Assume that
$S$ is CR in a small enough neighborhood $U$ of $p\in S$. If all CR
orbits of
$S$ are $1$-codimensional (which will appear as a necessary condition
for our problem), the following two conditions are equivalent \cite{DTZ05}:

$(i)$ $S$ is a locally flat boundary on $U$;

$(ii)$ $S$ is nowhere minimal on $U$.

\subsection{Complex points of $S$}\label{Sec:2.2} (i.e. singular CR points on $S$) \cite{DTZ05}.

At such a point
$p\in S$, $T_pS$ is a complex hyperplane in $T_p\C^n$. In suitable local
holomorphic coordinates $(z,w)\in \C^{n-1}\times\C$ vanishing at $p$, with $w=z_n$ and $z=(z_1,\ldots,z_{n-1})$, $S$ is locally given by the equation
$$ w= Q(z) + O(|z|^3),\quad
Q(z)= \sum_{1\leq i,j\leq n-1} (a_{ij}z_iz_j + b_{ij}z_i\overline z_j  +c_{ij}\overline z_i\overline z_j) \leqno (1)$$

$S$ is said {\it flat} at a
complex point $p\in S$ if $\sum b_{ij}z_i\overline z_j\in\lambda{\bf R},
\lambda\in {\C}$. We also say that $p$ {\it is flat}.

{\it Let $S\subset\C^n$ be a
locally flat boundary with a complex point $p$. Then $p$ is flat.}

By making the change of coordinates $(z,w)\mapsto(z,\lambda^{-1}
w)$, we get $\sum b_{ij}z_iz_j\in\R$ for all $z$.
By a change of coordinates $(z,w)\mapsto (z,w+\sum
a'_{ij}z_iz_j)$ we can choose the holomorphic term in (1) to be
the conjugate of the antiholomorphic one and so make the whole form
$Q$ real-valued.

We say that $S$ is in a {\it flat normal form} at $p$
if the coordinates $(z,w)$ as in (1) are chosen
such that $Q(z)\in{\bf R}$ for all $z\in\C^{n-1}$.

\subsubsection{Properties of $Q$.}\label{Sec:2.2.1}

Assume that $S$ is in a flat normal form; then, the quadratic form $Q$ is
real valued. If $Q$ is positive definite or negative definite, the point
$p\in S$ is said to be {\it
elliptic}; if the point $p\in S$ is not elliptic, and if $Q$ is non
degenerate, $p$ is said to be {\it hyperbolic}. From section \ref{Sec:2.4}, we will
only consider particular cases of the quadratic form $Q$.

\subsection{Elliptic points}\label{Sec:2.3}

\subsubsection{Properties of $Q$.}\label{Sec:2.3.1}

\begin{prop} (\cite{DTZ05,DTZ10}).\label{Prop:2.3.1}
Assume that $S\subset\C^n$, ($n\ge 3$) is nowhere minimal at all
its CR points
and has an elliptic flat complex point $p$. Then there exists a
neighborhood $V$ of $p$ such that $V\setminus \{p\}$ is foliated by
compact real $(2n-3)$-dimensional CR orbits diffeomorphic to the
sphere ${\bf S}^{2n-3}$ and there exists a smooth function $\nu$,
having the CR orbits as the level surfaces.
\end{prop}

\begin{proof}[Sketch of Proof] (see \cite{DTZ10}).
In the case of a quadric $S_0$ ($w=Q(z)$), the CR orbits are defined by
$w_0=Q(z)$, where $w_0$ is constant. Using (1), we approximate the
tangent space to $S$ by the tangent space to $S_0$ at a point with the
same coordinate $z$; the same is done for the tangent spaces to the CR
orbits on $S$ and $S_0$; then we construct the global CR orbit on $S$
through any given point close enough to $p$.
\end{proof}

\subsection{Special flat complex points}\label{Sec:2.4}

From \cite{Bi65}, for $n=2$, in suitable local holomorphic coordinates centered at 0, \ \
$Q(z)=(z\overline z+\lambda Re \
z^2), \ \ \lambda\geq 0 $, under the notations of \cite{BeK91}; for
$0\leq\lambda<1$, $p$ is said to be {\it elliptic}, and for $1<\lambda$,
it is said to be {\it hyperbolic}. The parabolic case $\lambda=1$, not
generic, will be omitted \cite{BeK91}. When $n \geq 3$, the Bishop's
reduction cannot be generalized.

\smallskip

We say that the flat complex
point
$p\in S$ is {\it special} if in convenient holomorphic coordinates centered at 0,
$$Q(z)=\sum_{j=1}^{n-1} (z_j\overline z_j+\lambda_j Re \ z_j^2), \ \
,\lambda_j\geq 0 \leqno (2)$$

Let $z_j=x_j+iy_j, \ x_j, y_j$ real, $j=1,\ldots,n-1$, then:

\smallskip

\noindent(3)\hskip 5mm
$Q(z)=\sum_{l=1}^{n-1}
\big((1+\lambda_l)x_l^2+(1-\lambda_l)y_l^2\big)+ O(|z|^3)$.

\smallskip

A flat point $p\in S$ is said to be {\it special elliptic} if
$0\leq\lambda_j<1$ for any $j$.

A flat point
$p\in S$ is said to be {\it special
k-hyperbolic} if $1<\lambda_j$ for $j\in J\subset \{1,\ldots,n-1\}$ and \
$0\leq\lambda_j<1$ \ for $j\in \{1,\ldots,n-1\}\setminus J\not
=\emptyset$, where
$k$ denotes
the number of elements of $J$.

\smallskip

{\it Special elliptic (resp. special $k$-hyperbolic) points are elliptic
(resp. hyperbolic).}

\smallskip

Special flat complex points

\subsection{Special hyperbolic points}\label{Sec:2.5}

$S$ being given by (1), let $S_0$ be the quadric of equation
$w=Q(z)$.
\begin{lem}\label{Lem:2.5.1}
Suppose that $S_0$ is flat at $0$ and that $0$ is a special $k$-hyperbolic point.
Then, in a neighborhood of \ $0$, and with the above local coordinates, $S_0$ is CR and
nowhere minimal outside $0$, and the CR orbits of $S_0$ are the $(2n-3)$-dimensional
submanifolds given by $w=const. \not =0$.
\end{lem}
\begin{proof} The submanifolds $w=const. \not =0$ have the same complex tangent space as $S_0$ and are of minimal dimension among submanifolds having this property, so they are CR orbits of codimension $1$, and from the end of section \ref{Sec:2.1}, $S_0$ is nowhere minimal outside 0.

The section $w=0$ of $S_0$ is a real quadratic cone $\Sigma'_0$ in
${\bf R}^{2n}$ whose vertex is 0 and, outside 0, it is a CR orbit
$\Sigma_0$ in the neighborhood of 0. We will improperly call $\Sigma'_0$ a
{\it singular CR orbit}.
\end{proof}

\subsection{Foliation by {\rm CR}-orbits in the neighborhood of a special $1$-hyperbolic point}\label{Sec:2.6}

We first mimic and transpose the begining of the proof of Proposition \ref{Prop:2.3.1}, i.e. of 2.4.2. in (\cite{DTZ05,DTZ09}).

\subsubsection{Local {\rm 2}-codimensional submanifolds}\label{Sec:2.6.1}

In order to use simple notations, we will assume $n=3$.

In $\C^3$,
consider the 4-dimensional submanifold
$S$ locally defined by the equation
$$w=\varphi(z)=Q(z)+O(\vert z\vert^3)\leqno (1)$$
and the 4-dimensional submanifold $S_0$ of equation
$$w=Q(z)\leqno (4)$$ with
$$Q=(\lambda_1+1)x_1^2-(\lambda_1-1)y_1^2+(1+\lambda_2)x_2^2+(1-\lambda_2)
y_2^2$$
having a special 1-hyperbolic point at 0, $(\lambda_1>1, 0\leq
\lambda_2<1)$,
and the cone $\Sigma'_0$ whose equation is: $Q=0$.
On $S_0$, a CR orbit is the 3-dimensional submanifold ${\mathcal K}_{w_0}$
whose equation is $w_0=Q(z)$. If $w_0>0$, ${\mathcal K}_{w_0}$ does not cut
the line $L=\{x_1=x_2=y_2=0\}$; if $w_0<0$, ${\mathcal K}_{w_0}$ cuts $L$ at
two points.
\begin{lem}\label{Lem:2.6.1'} $\Sigma_0=\Sigma'_0\setminus 0$ has two
connected components in a neighborhood of~0.
\end{lem}
\begin{proof} The equation of $\Sigma'_0\cap\{y_1=0\}$ is

$(\lambda_1+1)x_1^2+(1+\lambda_2)x_2^2+(1-\lambda_2)y_2^2=0$ \
\ whose only zero , in the neighborhood of $0$, is $\{ 0\}$: the connected
components are obtained for $y_1>0$ and
$y_1<0$ respectively.
\end{proof}

Local {\rm 2}-codimensional submanifolds

\subsubsection{{\rm CR}-orbits}\label{Sec:2.6.2}

By differentiating (1), we get for the tangent spaces the following
asymptotics
$$T_{(z,\varphi (z)})S=T_{(z,Q (z)})S_0+O(\vert z\vert^2),\hskip 3mm
z\in\C^2\leqno (5)$$
Here both $T_{(z,\varphi (z)})S$ and $T_{(z,Q (z)})S_0$ depend
continuously on $z$ near the origin.

\smallskip

Consider

$(i)$ the hyperbolo\"{\i}d $H_-=\{Q=-1\}$, (then $\displaystyle Q(\frac{z}{(-Q(z))^{1/2}})=-1$), and the projection:
$$\pi_-:\C^3\setminus\{z=0\}\rightarrow H_-,\hskip 3mm (z,w)\mapsto
\frac{z}{(-Q(z))^{1/2}}, $$

$(ii)$ for every $z\in H_-$, a real orthonormal basis $e_1(z),\ldots,
e_6(z)$ of $\C^3\cong\R^6$ such that
$$e_1(z),e_2(z)\in H_zH_-, \hskip 3mm e_3(z)\in T_zH_-,$$
where $HH_-$ is the complex tangent bundle to $H_-$.

Locally such a basis can be chosen continuously depending on $z$. For
every
$(z,w)\in \C^3\setminus\{z=0\}$, consider the basis
$e_1(\pi_-(z,w)),\ldots, e_6(\pi_-(z,w))$. The unit vectors
$e_1(\pi_-(z,w_0)),e_2(\pi_-(z,w_0)),e_3(\pi_-(z,w_0))$ are tangent to
the CR orbit ${\mathcal K}_{w_0}$ in $(z,w_0)$ for $w_0<0$.
Then, from (5), we have:
$$H_{(z,\varphi (z)})S=H_{(z,Q (z)})S_0+O(\vert z\vert^2),\hskip 3mm
z\not =0,\hskip 2mm z\rightarrow 0.\leqno (6)$$

\smallskip

As in {\rm \cite{DTZ10}}, in the neighborhood of 0, denote by $E(q), q\in S\setminus
\{0\}, w<0$ the tangent space to the local CR orbit ${\mathcal K}$ on $S$ through
$q$, and by $E_0(q_0), q_0\in S_0\setminus \{0\}, w<0$ the analogous
object for $S_0$. We have :

$$E(z,\varphi(z))=E_0(z,Q(z))+O(\vert z\vert^2),\hskip 2mm z\not =0,
\hskip 2mm z\rightarrow 0\leqno (7)$$

Given $\underline q\in S$, by integration of $E(q)$, $q\in S$, we get,
locally, the CR orbit (the leaf), on $S$ through $\underline q$; given
$\underline q_0\in S_0$, by integration of
$E_0(q_0)$, $q_0\in S_0$, we get, locally, the CR orbit (the leaf), on
$S_0$ through $\underline q_0$ (theorem of Sussman). On $S_0$, a leaf
is the 3-dimensional submanifold ${\mathcal K}_{\underline q_0}=
{\mathcal K}_{w_0}={\mathcal K}_0 $ whose
equation is $w_0=Q(z)$, with $\underline q=(z_0,w_0=Q(z_0))$. $d\pi_-$
projects each $E_0(q)$, $q\in S_0, w<0$, bijectively onto
$T_{\pi(q)}H_-$, then
$\pi_-\vert_{{\mathcal K}_0}$ is a diffeomorphism onto $H_-$; this implies,
from (7), that, in a suitable neighborhood of the origin, the restriction
of $\pi_-$ to each local CR orbit of $S$ is a local
diffeomorphism.

We have: $\varphi(z)=Q(z)+\Phi (z)$ with $\Phi (z)=O(\vert
z\vert^3)$.

\subsubsection{Behaviour of local CR orbits}\label{Sec:2.6.3}

Follow the construction of
$E(z,\varphi(z))$; compare with
$E_0(z,Q(z))$. We know the integral manifold, the orbit of $E_0(z,Q(z))$;
deduce an evaluation of the integral manifold ${\mathcal K}$ of
$E(z,\varphi(z))$.
\begin{lem}\label{Lem:2.6.3'} Under the above hypotheses,
the local orbit $\Sigma$ corresponding to $\Sigma_0$
has two connected components in the neighborhood of $0$.
\end{lem}

\begin{proof}
Using the real coordinates, as for Lemma \ref{Lem:2.6.1'}, $\Sigma'\cap\{y_1=0\}$. Locally, the connected components are obtained for $y_1>0$ and $y_1<0$ respectively, from formula $(1)$.
\end{proof}

We will improperly call $\Sigma'=\overline \Sigma$ \ \ a {\it singular CR
orbit} and a {\it singular leaf of the foliation}.

\smallskip

We intend to prove: 1) ${\mathcal K}$ does not cross the singular leaf through 0;

\hskip 3.2cm 2) the only separatrix is the singular leaf through 0.

\smallskip

From the orbit ${\mathcal K}_0$, construct the differential equation defining
it, and using (7), construct the differential equation defining ${\mathcal K}$.

\smallskip

In $\C^3$, we use the notations: $x=x_1, y=y_1, u=x_2, v=y_2$; it suffices to consider the particular case:
$Q=3x^2-y^2+u^2+v^2$. On $S_0$, the orbit ${\mathcal K}_0$ issued from the point
$(c,0,0,0)$ is defined by: $3x^2-y^2-u^2+v^2=3c^2$, i.e., for $x\geq 0$, $x=\frac {1}{\sqrt
3}(y^2-u^2-v^2+3c^2)^{\frac{1}{2}}=A(y,u,v)$; the local coordinates on the orbit are $(y,u,v)$. ${\mathcal K}_0$ satisfies the
differential equation: $dx=dA$. From (9), the orbit ${\mathcal K}$, issued from
$(c,0,0,0)$, satisfies $dx=dA+\Psi$
with $\Psi(y,u,v;c)=O(\vert
z\vert^2)$; hence $\Psi=d\Phi$, then $x=A+\Phi$, with $\Phi=O(\vert
z\vert^3)$. More explicitly, ${\mathcal K}$ is defined by:
$$x=x_{{\mathcal K},c}={\frac {1}{\sqrt
3}}(y^2-u^2-v^2+3c^2)^{\frac{1}{2}}+\Phi(y,u,v;c),
{\hskip 2mm} \Phi(y,u,v;c)=O({\vert z\vert}^3)$$

\bigskip

The cone
$\Sigma'_0$ whose equation is: $Q=0$ is a separatrix for the orbits ${\mathcal
K}_0$. The corresponding object $\Sigma'=\{\varphi(z)=0\}$ for $S$ has the singular point
$0$ and for
$x>0, y>0$, $u>0, v>0$ is defined by the differential equation
$dx=d(A+\Phi)$, with
$c=0$, i.e. the local equation of $\Sigma'$ is
$$x=x_{{\mathcal K},0}=\frac {1}{\sqrt
3}(y^2-u^2-v^2)^{\frac{1}{2}}+\Phi(y,u,v;0), \hskip 2mm
\Phi(y,u,v;0)=O(\vert z\vert^3)$$

For given $(y,u,v)$, $x_{{\mathcal K},c}-x_{{\mathcal K},0}=
x_{{\mathcal K}_0,c}-x_{{\mathcal
K}_0,0}+\Phi(y,u,v;c)-\Phi(y,u,v;0)$. But $x_{{\mathcal K}_0,c}-x_{{\mathcal
K}_0,0}=O(1)$ and $\Phi(y,u,v;c)-\Phi(y,u,v;0)=O(\vert z\vert^3)$.

\smallskip

As a consequence, for $x>0, y>0, u>0,v>0$, locally, $\Sigma'$ is a separatrix for
the orbits ${\mathcal K}$, and the only one. Same result for $x<0$.

\subsubsection{}\label{Sec:2.6.4} What has been done from the hyperbolo\"{\i}d $H_-=\{Q=-1\}$
can be repeated from the hyperbolo\"{\i}d $H_+=\{Q=1\}$.

\smallskip

\noindent As at the beginning of the section \ref{Sec:2.6.2}, we consider

$(i)$ the hyperbolo\"{\i}d $H_+\{Q=1\}$ and the projection:
$$\pi_+:\C^3\setminus\{z=0\}\rightarrow H_+,\hskip 3mm (z,w)\mapsto
\frac{z}{(Q(z))^{1/2}}, $$

$(ii)$ for every $z\in H_+$, a real orthonormal basis $e_1(z),\ldots,
e_6(z)$ of $\C^3\cong\R^6$ such that
$$e_1(z),e_2(z)\in H_zH_+, \hskip 3mm e_3(z)\in T_zH_+,$$
where $HH_+$ is the complex tangent bundle to $H_+$.

\subsubsection{}\label{Sec:2.6.5}

\begin{lem}\label{Lem:2.6.5} Given $\varphi$, there exists $R>0$
such that, in $B(0,R)\cap \{x>0, y>0, u>0,v>0\}\subset \C^2$, the CR
orbits
${\mathcal K}$ have $\Sigma'$ as unique separatrix.
\end{lem}
\begin{proof} When c tends to zero, , $x_{{\mathcal K},c}-x_{{\mathcal K},0}=x_{{\mathcal K}_0,c}-x_{{\mathcal
K}_0,0}=O(\vert z\vert)$, $\Phi(y,u,v;c)-\Phi(y,u,v;0)=O(\vert z\vert^3)$.
For $\varphi(z)=Q(z)+\Phi (z)$ with $\Phi (z)=O(\vert z\vert^3)$ given, in
(9), $E(z,\varphi(z))-E_0(z,Q(z))=O(\vert z\vert^2)$ and
$\Phi(y,u,v;c)-\Phi(y,u,v;0)=O(\vert z\vert^3)$ are also given. Then there
exists $R$ such that, for $\vert z\vert < R$, $x_{{\mathcal K},c}-x_{{\mathcal
K},0}>0$.
\end{proof}

\subsection{{\rm CR}--orbits near a subvariety containing a special $1$-hyperbolic point}\label{Sec:2.7}

\subsubsection{}\label{Sec:2.7.1}
In the section \ref{Sec:2.7}, we will impose conditions on $S$ and
give a local property in the neighborhood of a compact $(2n-3)$-subvariety
of $S$.

Assume that $S\subset\C^n$ $(n\geq 3)$, is a locally closed $(2n-2)$-submanifold, nowhere minimal at all its CR points, which has a unique $1$-hyperbolic flat complex point $p$, and such that:

$(i)$ $\Sigma$ being the orbit whose closure $\Sigma'$ contains $p$, then $\Sigma'$ is compact.

\smallskip

Let $q\in S$, $q\not=p$; then, in a neighborhood $U$ of $q$ disjoint from
$p$, $S$ is CR, CR-dim $S= n-2$, $S$ is non minimal and $\Sigma$ is
1-codimensional. To show that the CR orbits contitute a foliation on
$S$ whose separatrix is $\Sigma'$: this is true in $U$ since $\Sigma\cap U$
is a leaf. Moreover, let $U_0$ the ball $B(0,R)$ centered in $p=0$ in Lemma
\ref{Lem:2.6.5}, if $U\cap U_0\not=\emptyset$, the leaves in $U$ glue with the leaves
in $U_0$ on $U\cap U_0$. Since $\Sigma'$ is compact, there exists a finite number
of points $q_j\in\Sigma'$, $j=0,1,\ldots,J$, and open neighborhoods $U_j$, as
above, such that $(U_j)_{j=0}^J$ is an open covering of $\Sigma'$. Moreover the
leaves on $U_j$ glue respectively with the leaves on $U_k$ if $U_j\cap
U_h\not=\emptyset$.

\subsubsection{}\label{Sec:2.7.2}
\begin{prop}\label{Prop:2.7.2}
Assume that $S\subset\C^n$
$(n\geq 3)$, is a locally closed $(2n-2)$-submanifold, nowhere minimal at
all its CR points, which has a unique special $1$-hyperbolic flat complex
point $p$, and such that:

$(i)$ $\Sigma$ being the orbit whose closure $\Sigma'$ contains $p$, then $\Sigma'$ is compact;

$(ii)$ $\Sigma$ has two connected components $\sigma_1$,
$\sigma_2$, whose closures are homeomorphic to spheres of dimension $2n-3$.

\smallskip

Then, there exists a neighborhood \ $V$ of \ $\Sigma'$ such that \ $V\setminus
\Sigma'$ is foliated by compact real $(2n-3)$-dimensional CR orbits whose equation,
in a neighborhood of \ $p$ is $(3)$, and, the $w(=x_n)$-axis being assumed to be
vertical, each orbit is diffeomorphic to

the sphere ${\bf S}^{2n-3}$ above $\Sigma'$,

the union of two spheres ${\bf S}^{2n-3}$ under $\Sigma'$,

\noindent and there exists a smooth function $\nu$,
having the CR orbits as the level surfaces.
\end{prop}
\begin{proof} From subsection \ref{Sec:2.7.1} and the following
remark:

When $x_n$ tends to 0, the
orbits tends to $\Sigma'$, and because of the geometry of the orbits near $p$, they
are diffeomorphic to a sphere above $\Sigma'$, and to the
union of two spheres under $\Sigma'$. The existence of $\nu$ is proved as in
Proposition \ref{Prop:2.3.1}, namely, consider a smooth curve $\gamma: \lbrack 0,\varepsilon)\rightarrow S$ such that $\gamma (0)=q$, where $q$ is a point of $\Sigma$ close to $p$, and $\gamma$ is a diffeomorphism onto its image $\Gamma=\gamma(\lbrack 0,\varepsilon))$. Let $\nu=\gamma^{-1}$ on the image of $\gamma$, then, close enough to $q$, every CR orbit cuts $\Gamma$ at a unique point $q(t)$, $t\in \lbrack 0,\varepsilon)$. Hence there is a unique extension of $\nu$ from $\gamma(\lbrack 0,\varepsilon))$ to $V\setminus p$ where $V$ is a neighborhood of $\Sigma'$ having CR orbits as its level surfaces. $\nu$ being smooth away from $p$, it is smooth on the orbit $\Sigma$ and, if we set $\nu(p)=\nu(q)=0$, $\nu$ is smooth on a neighborhood of $\Sigma\cup \{p\}=\Sigma'$.

\end{proof}

\subsection{Geometry of the complex points of $S$}\label{Sec:2.8}

The results of section \ref{Sec:2.8} are particular cases of theorems of H-F Lai \cite{L72}, that I learnt from F. Forstneric in July 2011.

In \cite{BeK91} E. Bedford \& W. Klingenberg cite the
following theorem of E. Bishop \cite{Bi65}[section 4, p.15]: {\it On a
$2$-sphere embedded in $\C^2$, the difference between the numbers of
elliptic points and of hyperbolic points is the Euler-Poincar\'e
characteristic, i.e.} 2. For the proof, Bishop uses a theorem of ([CS
51], section 4).

We extend the result for $n\geq 3$ and give proofs which are essentially the same than in the general case of \cite{L72,L74} but simpler.

\subsubsection{}\label{Sec:2.8.1}
Let $S$ be a smooth compact connected oriented submanifold of dimension $2n-2$. Let $G$ be the manifold of the oriented real linear
$(2n-2)$-subspaces of $\C^n$. The submanifold $S$ of $\C^n$ has a given orientation
which defines an orientation $o(p)$ of the tangent space to $S$ at any point $p\in
S$. By mapping each point of $S$ into its oriented tangent space, we get a smooth
Gauss map
$$t:S\rightarrow G$$

Denote $-t(p)$ the tangent space to $S$ at $p$ with opposite orientation $-o(p)$.

\subsubsection{Properties of $G$}\label{Sec:2.8.2}

$(a)$ dim $G=2(2n-2)$.

\begin{proof} $G$ is a two-fold covering of the Grassmannian $M_{m,k}$,
of the linear $k$-subspaces of $\R^m$ \cite{S51}[Part, section 7.9], for $m=2n$ and $k=2n-2$; they have
the same dimension. We have:
$$M_{m,k}\cong O_m/O_k\times O_{m-k}$$

But dim $O_k=\displaystyle\frac {1}{2} k(k-1)$, hence dim $M_{m,k}= \displaystyle \frac {1}{2} \Big (m(m-1)-k(k-1)-(m-k)(m-k-1)\Big)=k(m-k)$.

\smallskip

$(b)$ $G$ has the complex structure of a smooth quadric of complex dimension $(2n-2)$ of $\C P^{2n-1}$ L74, \cite{P08}.

\smallskip

$(c)$ There exists a canonical isomorphism $h: G\rightarrow \C P^{n-1}\times\C P^{n-1}$.

\smallskip

$(d)$ Homology of $G$ (cf \cite{P08}): Let $S_1,S_2$ be generators of $H_{2n-2}(G,\Z)$; we assume
that $S_1$ and
$S_2$ are fundamental cycles of complex projective subspaces of complex dimension $(n-1)$ of the complex quadric $G$. We also denote $S_1,S_2$ the ordered two factors $\C P^{n-1}$, so that $h:G\rightarrow S_1\times S_2$.

\end{proof}

\subsubsection{}\label{Sec:2.8.3}
\begin{prop}\label{Prop:2.8.3}
For $n\geq 2$, in general, $S$ has
isolated complex points.
\end{prop}
\begin{proof}
Let $\pi\in G$ be a complex hyperplane of $\C^n$ whose
orientation is induced by its complex structure; the set of such $\pi$ is
$H=G_{n-1,n}^{\C}=\C{\rm P}^{n-1*}\subset G$, as real submanifold. If $p$ is a
complex point of $S$, then $t(p)\in H$ or $-t(p)\in H$. The set of complex points of
$S$ is the inverse image by
$t$ of the intersections
$t(S)\cap H$ and $-t(S)\cap H$ in $G$. Since dim $t(S)= 2n-2$, dim $H=2(n-1)$, dim $G=
2(2n-2)$, the intersection is 0-dimensional, in general.
\end{proof}

\subsubsection{}\label{Sec:2.8.4}
Denoting also $S$, the fundamental cycle
of the submanifold $S$ and $t_*$ the homomorphism defined by $t$, we have:
$$t_*(S)\sim u_1S_1+u_2S_2$$where $\sim$ means {\it homologous to}.

\subsubsection{}\label{Sec:2.8.5}
\begin{lem}[proved for $n=2$ in \cite{CS51}]\label{Lem:2.8.5} With the above
notations, we have:
\ \ $u_1=u_2$;
\ \ $u_1+u_2=\chi(S)$, Euler-Poincar\'e characteristic of $S$.
\end{lem}

The proof for $n=2$ works for any $n\geq 3$, namely:

\smallskip

Let $G'$ be the manifold of the oriented real linear
$2$-subspaces of $\C^n$. Let $\alpha:G\rightarrow G'$ map each oriented $2(n-1)$-subspace $R$ onto its normal 2-subspace $R'$ oriented so that $R,R' $ determine the orientation of $\C^n$.
\ \ \ $\alpha$ is a canonical isomorphism. Let $n: S\rightarrow G'$ the map defined by taking oriented normal planes; then: $n=\alpha t$ and $t=\alpha^{-1}n$, hence the mapping $h\alpha h^{-1}: S_1\times S_2\rightarrow S_1\times S_2$. Let $(x,y)$ be a point of $S_1\times S_2$, then \hskip 3mm $(\dag ) \ \ \ h\alpha h^{-1}(x,y)=(x,-y)$.

Over $G$, there is a bundle $V$ of spheres obtained by considering as fiber over a real oriented linear $(2n-2)$-subspace of $\C^n$ through 0 the unit sphere ${\bf S}^{2n-3}$ of this subspace. Let $\Omega$ be the characteristic class of $V$, and let $\Omega_t$, $\Omega_n$ denote the characteristic classes of the tangent and normal bundles of $S$. Then \hskip 3mm $t^*\Omega=\Omega_t$, $n^*\Omega=\Omega_n$.

$V$ is the Stiefel manifold of ordered pairs of orthogonal unit vectors through in $\R^{2n}\cong\C^n$. Let $f:V\rightarrow G$ the projection.

From the Gysin sequence, we see that the kernel of \hskip 3mm $f^*:H^{2n-2}(G)\rightarrow H^{2n-2}(V)$ is generated by $\Omega$. To find the kernel of $f^*$, we determine the morphism \hskip 3mm $f_*: H_{2n-2}(V)\rightarrow H_{2n-2}(G)$. A generating $2n-2)$-cycle of in $V$ is $S^2\times e$ where $S^2\cong\C P^{n-1}$ and $e$ is a point. Let $z$ be any point of $S^2$, then
from $(\dag )$, we have
$$hf(z,e)=(z,-z)$$
Therefore, we see that $f_*(S^2\times e)=S_1-S_2$. Then, the kernel of $f^*$ is $\Z$-generated by $S_1^*+S_2^*$.

With convenient orientation for the fibre of the bundle $V$, we get: $\Omega= S_1^*+S_2^*$. For convenient orientation of $S$, we get $\Omega_t.S=\chi_S=$ Euler characteristic of $S$. We have
$$\Omega_t=t^*(S_1^*+S_2^*)=t^*S_1^*+t^*S_2^*$$

$$\Omega_n=n^*(S_1^*+S_2^*)=t^*\alpha^*(S_1^*+S_2^*)=t^*(S_1^*-S_2^*)=t^*S_1^*-t^*S_2^*$$
Since $\Omega_n=0$, we get:
$$(t^*S_1^*).S=(t^*S_2^*).S=\frac{1}{2}\chi_S$$

\subsubsection{Local intersection numbers of $H$ and $t(S)$ when all
complex points are flat and special}\label{Sec:2.8.6}

\noindent $H$ is a complex linear $(n-1)$-subspace of $G$, then is homologous to
one of the $S_j$, $j=1, 2$, say $S_2$ when $G$ has its structure of complex quadric.
The intersection number of $H$ and $S_1$ is 1 and the intersection number of $H$ and
$S_2$ is 0. So, the intersection number of $H$ and $u_1S_1+u_2S_2$ is $u_1$.

\noindent In the neighborhood of a complex point $0$, $S$ is
defined by equation (1), with $w=z_n$ and
$$Q(z)=\sum_{j=1}^{n-1}\mu_j (z_j\overline z_j+\lambda_j {\mathcal R}e \ z_j^2), \ \ \mu_j>0,
\lambda_j\geq 0 \leqno (1')$$
Let $z_j=x_{2j-1}+ix_{2j},\ \ j=1,\ldots,n,$ with real $x_l$. \ \ Let $e_l$ the
unit vector of the $x_l$ axis, $l=1,\ldots,2n$.

\smallskip

For simplicity assume $n=3$:\ \ $Q(z)=\mu_1 (z_1\overline
z_1+\lambda_1 {\mathcal R}e \ z_1^2)+\mu_2 (z_2\overline
z_2+\lambda_2 {\mathcal R}e \ z_2^2)$, with $\mu_1=\mu_2=1$.

Then, up to higher order terms, $S$ is defined by:

$z_1=x_1+ix_2; \ \ z_2=x_3+ix_4;
\ \ z_3=(1+\lambda_1)x_1^2+(1-\lambda_1)x_2^2+(1+\lambda_2)x_3^2+(1-\lambda_2)x_4^2$.

In the neighborhood of $0$, the tangent space to $S$ is defined by the four
independent vectors 
\begin{multline*}
 \nu_1=e_1+2(1+\lambda_1)x_1\ e_5;\ \
\nu_2=e_2+2(1-\lambda_1)x_2\ e_5;\ \
\nu_3=e_3+2(1+\lambda_2)x_3\ e_5;
\\ \nu_4=e_4+2(1-\lambda_2)x_4\ e_5
\end{multline*}

Then, if $0$ is special elliptic or special $k$-hyperbolic with $k$ even, the tangent
plane at $0$ has the same orientation; if $0$ is special elliptic or special $k$-hyperbolic
with $k$ odd the tangent space has opposite orientation.

\subsubsection{}\label{Sec:2.8.7}
\begin{prop}[known for $n=2$ \cite{Bi65}, {\it here for $n\geq 3$}] Let $S$ be a smooth, oriented, compact, 2-codimensional, real
submanifold of $\C^n$ whose all complex points are flat and special elliptic or special 1-hyperbolic. Then, on $S$, $\sharp$
(special elliptic points) - $\sharp$
(special $1$-hyperbolic points $=\chi(S)$. If $S$ is a
sphere, this number is 2.
\end{prop}
\begin{proof}
Let $p\in S$ be a complex point and $\pi$ be the
tangent hyperplane to $S$ at $\pi$. Assume that

(**) {\it the orientation of $S$
induces, on $\pi$, the orientation given by its complex structure},

\noindent then $\pi\in H$.

If $p$ is elliptic, the
intersection number of $H$ and $t(S)$ is 1; if $p$ is $1$-hyperbolic,
the intersection number of $H$ and $t(S)$ is -1 at $p$.

From the beginning of section \ref{Sec:2.8.6}, the sum of the intersection numbers of $H$ and $t(S)$ at
complex points $p$ satisfying (**) is $u_1$. Reversing the condition
(**), and using Lemma \ref{Lem:2.8.5}, we get the Proposition.
\end{proof}

\section{Particular cases: horned sphere; elementary models and their gluing}\label{Sec:3}

\subsection{}\label{Sec:3.1}

We recall the following Harvey-Lawson theorem with real parameter to be used later.

\subsubsection{}\label{Sec:3.1.1}
Let $E\cong{\bf R}\times\C^{n-1}$, and $k:{\bf
R}\times\C^{n-1}\rightarrow {\bf R}$ be the projection.
Let $N\subset E$ be a
compact, (oriented) CR subvariety of $\C^{n+1}$ of
real dimension $2n-2$ and CR dimension
$n-2$, $(n\geq 3)$, of class $C^\infty$, with negligible
singularities (i.e. there exists a closed subset $\tau\subset N$ of
$(2n-2)$-dimensional Hausdorff measure $0$
such that $N\setminus \tau$ is a CR submanifold).
Let $\tau'$ be the set of all points $z\in N$
such that either $z\in\tau$ or $z\in N\setminus\tau$
and $N$ is not transversal to the complex hyperplane
$k^{-1}(k(z))$ at $z$.
Assume that $N$, as a
current of integration, is $d$-closed and satisfies:

{\rm (H)} there exists a closed subset $L\subset{\R}_{x_1}$
with
$ H^1(L)=0$ such that for every $x\in k(N)\setminus L$,
the fiber $k^{-1}(x)\cap N$ is connected
and does not intersect $\tau'$.

\subsubsection{}\label{Sec:3.1.2}
\begin{thm}[\cite{DTZ10} (see also \cite{DTZ05})]\label{Thm:3.1.2} Let $N$
satisfy {\rm (H)} with
$L$ chosen accordingly. Then, there exists, in $E'= E\setminus k^{-1}(L)$, a
unique
$ C^\infty$ Levi-flat $(2n-1)$-subvariety $M$ with negligible singularities in
$E'\setminus N$, foliated by complex $(n-1)$-subvarieties, with the properties
that
$M$ simply (or trivially) extends
to
$E'$ as a $(2n-1)$-current (still denoted $M$) such that $dM=N$ in
$E'$.1
The leaves are the sections by the hyperplanes $E_{x_1^0}$, $x_1^0\in
k(N)\setminus L$, and are the solutions of the ``Harvey-Lawson problem'' for
finding a holomorphic subvariety in $E_{x_1^0}\cong\C^n$ with prescribed boundary $N\cap E_{x_1^0}$.
\end{thm}

\subsubsection{}\label{Sec:3.1.3}
\begin{rem}\label{Rem:3.1.3}
Theorem~\ref{Thm:3.1.2} is valid in the space $E\cap \{\alpha_1< x_1<\alpha_2\}$, with the corresponding condition (H). Moreover, since $N$ is compact, for convenient coordinate $x_1$, we can assume $x_1\in [0,1]$.
\end{rem}

\subsection{}\label{Sec:3.2}
{\it To solve the boundary problem by Levi-flat hypersurfaces, $S$ has to satisfy necessary
and sufficient local conditions. A way to prove that these conditions can
occur is to construct an example for which the solution is obvious.}

\subsection{Sphere with one special 1-hyperbolic point (sphere with two
horns): Example}\label{Sec:3.3}

\subsubsection{}\label{Sec:3.3.1}
In $\C^3$, let $(z_j)$, $j=1,2,3$, be the complex
coordinates and $z_j=x_j+iy_j$. In ${\bf R}^6\cong\C^3$, consider the
4-dimensional subvariety (with negligible singularities) $S$ defined
by:

$y_3=0$

$0\leq x_3\leq 1$; \hskip 2mm
$x_3(x_1^2+y_1^2+x_2^2+y_2^2+x_3^2 -1)+(1-x_3)(x_1^4+y_1^4+x_2^4+y_2^4+4x_1^2
-2y_1^2+x_2^2+y_2^2)=0$

$-1\leq x_3\leq 0$; \hskip 2mm$x_3=x_1^4+y_1^4+x_2^4+y_2^4+4x_1^2-2y_1^2+x_2^2+y_2^2$

\smallskip

The singular set of $S$ is the 3-dimensional section $x_3=0$ along which the
tangent space is not everywhere (uniquely) defined. $S$ being in the real hyperplane $\{y_3=0\}$, the complex tangent spaces to
$S$ are $\{x_3=x^0\}$ for convenient $x^0$.

\subsubsection{}\label{Sec:3.3.2}
The tangent space to the hypersurface $f(x_1,y_1,x_2,y_2,x_3)=0$ in $\R^5$ is

$$X_1f'_{x_1}+Y_1f'_{y_1}+X_2f'_{x_2}+Y_2f'_{y_2}+X_3f'_{x_3}=0,$$
Then, the tangent space to $S$ in the hyperplane $\{y_3=0\}$ is:

for $0\leq x_3$,

\begin{multline*}
2x_1\lbrack x_3+2(1-x_3)(x_1^2+2) \rbrack X_1+2y_1\lbrack x_3+2 (1-x_3)(y_1^2-1)\rbrack Y_1\\+2x_2\lbrack x_3+(1-x_3)(2x_2^2+1) \rbrack X_2+2y_2\lbrack x_3+(1-x_3)(2y_2^2+1)\rbrack Y_2
\\+\lbrack (x_1^2+y_1^2+x_2^2+y_2^2+3x_3^2 -1)\hspace*{2cm}\\ -(x_1^4+y_1^4+x_2^4+y_2^4+4x_1^2-2y_1^2+x_2^2+y_2^2)\rbrack X_3=0;
\end{multline*}

for $x_3\leq 0$,

$$4(x_1^2+2)x_1X_1 +4(y_1^2-1)y_1Y_1 +2(2x_2^2+1)x_2X_2 +2(2y_2^2+1)y_2Y_2-X_3=0.$$

\subsubsection{}\label{Sec:3.3.3}
The complex points of $S$ are defined by the vanishing of
the coefficients of $X_j$, j=1,2,3,4 in the equation of the tangent spaces

\noindent for $0\leq x_3\leq 1$,

\smallskip

$x_1\lbrack x_3+2(1-x_3)(x_1^2+2)\rbrack=0$,

$y_1\lbrack x_3+2 (1-x_3)(y_1^2-1)\rbrack=0$,

$x_2\lbrack x_3+(1-x_3)(2x_2^2+1) \rbrack =0$,

$y_2\lbrack x_3+(1-x_3)(2y_2^2+1)\rbrack=0$.

\smallskip

We have the solutions

\smallskip

\noindent $h$: $x_j=0,y_j=0$, $(j=1, 2)$, $x_3=0$;

\noindent $e_3$: $x_j=0,y_j=0$, $(j=1,2)$, $x_3=1$.

\noindent for $x_3\leq 0$,

\smallskip

$(x_1^2+2)x_1=0$,

$(y_1^2-1)y_1=0$,

$(2x_2^2+1)x_2=0$,

$(2y_2^2+1)y_2=0$.

\smallskip

We have the solutions

\smallskip

\noindent $h$: $x_j=0,y_j=0$, $(j=1, 2)$, $x_3=0$;

\noindent $e_1,e_2$: $x_1=0, y_1=\pm 1, x_2=0, y_2=0, x_3=-1$.

\smallskip

Remark that the tangent space to $S$ at $h$ is well defined.
Moreover, the set $S$ will be smoothed along its section by the hyperplane $\{x_3=0\}$ by a small deformation leaving $h$ unchanged. {\it In the following $S$ will denote this smooth submanifold}.

\subsubsection{}\label{Sec:3.3.4}
\begin{lem}\label{Lem:3.3.4}
The points $e_1,e_2,e_3$ are special elliptic;
the point $h$ is special $\{1\}$-hyperbolic.
\end{lem}

\begin{proof}
Point $e_3$: Let $x'_3=1-x_3$, then the equation of $S$ in the neighborhood of $e_3$ is:

$(1-x'_3)(x_1^2+y_1^2+x_2^2+y_2^2+x_3^{'2}-2x'_3)-x'_3(x_1^4+y_1^4+x_2^4+y_2^4+4x_1^2-2y_1^2+x_2^2+y_2^2)=0$, {\it i.e.}

$2x'_3=x_1^2+y_1^2+x_2^2+y_2^2)+ O(|z|^3)$, or $w=z\overline z + O(|z|^3)$

\noindent then $e_3$ is special elliptic.

\smallskip

\noindent Points $e_1,e_2$: Let $y'_1=y_1\pm 1$, $x'_3=x_3+1$, then the equation of $S$ in the neighborhood of
$e_1,e_2$ is:

$x'_3-1=x_1^4+(y'_1\mp 1)^4+x_2^4+y_2^4+4x_1^2-2(y'_1\mp 1)^2+x_2^2+y_2^2$

$=x_1^4+y_1^{'4}\mp 4y_1^{'3}+6y_1^{'2}\mp 4y'_1+1+x_2^4+y_2^4+4x_1^2-2(y'_1\mp 1)^2+x_2^2+y_2^2$, then

$x'_3=x_1^4+y_1^{'4}\mp 4y_1^{'3}+4y_1^{'2}+x_2^4+y_2^4+4x_1^2+x_2^2+y_2^2$, {\it i.e.}

$x'_3=4x_1^2+4y_1^{'2}+x_2^2+y_2^2+ O(|z|^3)$, or $w=4z_1\overline z_1+z_2\overline z_2$,

\noindent then $e_1,e_2$ are special elliptic.

\smallskip

\noindent Point $h$: The equation of $S$ in the neighborhood of $h$ is:

\noindent for $x_3\geq 0$,

\begin{multline*}
x_3(x_1^2+y_1^2+x_2^2+y_2^2+x_3^2 -1)\\+(1-x_3)(x_1^4+y_1^4+x_2^4+y_2^4+4x_1^2
-2y_1^2+x_2^2+y_2^2)=0
\end{multline*}

\noindent for $x_3\leq 0$,

$x_3=x_1^4+y_1^4+x_2^4+y_2^4+4x_1^2-2y_1^2+x_2^2+y_2^2$, {\it i.e.}

\smallskip

$x_3=4x_1^2-2y_1^2+x_2^2+y_2^2+ O(|z|^3)$, in both cases, up to the third order terms, {\it i.e.}:
$w=z_1\overline z_1+z_2\overline z_2+3 {\mathcal R}e \ z_1^2$,

\noindent then $h$ is special $\{1 \}$-hyperbolic.
\end{proof}

\subsubsection{Section $\Sigma'=S\cap\{x_3=0\}$}\label{Sec:3.3.5}
Up to a small
smooth deformation, its equation is:

\smallskip

$x_1^4+y_1^4+x_2^4+y_2^4+4x_1^2-2y_1^2+x_2^2+y_2^2=0$, in $\{x_3=0\}$.

\noindent The tangent cone to $\Sigma'$ at 0 is: $4x_1^2-2y_1^2+x_2^2+y_2^2=0$.

\smallskip

\noindent Locally, the section of $S$ by the coordinate 3-space

$x_1, y_1, x_3$ is: \hskip 3mm $x_3=4x_1^2-2y_1^2+O(|z|^3)$

$x_2, y_2, x_3$ is: \hskip 3mm $x_3 =x_2^2+y_2^2+O(|z|^3)$

\smallskip

\noindent 3.3.1'. {\it Shape of $\Sigma'=S\cap\{x_3=0\}$ in the neighborhood of
the origin
$0$ of $\C^3$}.

\begin{lem}\label{Lem:3.3.1'}
Under the above hypotheses and notations,

$(i)$ $\Sigma=\Sigma'\setminus 0$ has two connected components $\sigma_1$,
$\sigma_2$.

$(ii)$ The closures of the three connected components of $S\setminus\Sigma'$
are submanifolds with boundaries and corners.
\end{lem}
\begin{proof} $(i)$ The only singular point of $\Sigma'$ is 0. We
work in the ball $B(0,A)$ of $\C^2$ $(x_1, y_1,x_2, y_2)$ for small $A$ and
in the 3-space $\pi_\lambda = \{y_2=\lambda x_2\}$, $\lambda\in \R$. For
$\lambda$ fixed, $\pi_\lambda\cong \R^3(x_1, y_1,x_2)$, and
$\Sigma'\cap\pi_\lambda$ is the cone of equation
$4x_1^2-2y_1^2+(1+\lambda^2)x_2^2+O(|z|^3)=0$ with vertex 0 and basis in the plane
$x_2=x_2^0$ the hyperboloid
$H_\lambda$ of equation $4x_1^2-2y_1^2+(1+\lambda^2)x_2^{0 2}+O(|z|^3)=0$; the curves
$H_\lambda$ have no common point outside 0. So, when $\lambda$ varies, the surfaces
$\Sigma'\cap\pi_\lambda$ are disjoint outside 0. The set $\Sigma'$ is
clearly connected; $\Sigma'\cap\{y_1=0\}=\{0\}$, the origin of $\C^3$;
from above: $\sigma_1= \Sigma\cap\{y_1>0\}$; $\sigma_2=
\Sigma\cap\{y_1<0\}$.

$(ii)$ The three connected components of $S\setminus\Sigma'$ are the
components which contain, respectively $e_1$, $e_2$, $e_3$ and whose
boundaries are $\overline \sigma_1$, $\overline \sigma_2$, $\overline
\sigma_1\cup \overline \sigma_2$; these boundaries have corners as shown
in the first part of the proof.
\end{proof}

The connected component of $\C^2\times \R\setminus S$
containing the point $(0,0,0,0,1/2) $ is the Levi-flat solution, the complex
leaves being the sections by the hyperplanes $x_3=x_3^0$, $-1<x_3^0<1$.

The sections by the hyperplanes $x_3=x_3^0$ are diffeomorphic to a 3-sphere for
$0<x_3^0<1$ and to the union of two disjoint 3-spheres for $-1<x_3^0<0$, as can be
shown intersecting $S$ by lines through the origin in the
hyperplane $x_3=x_3^0$;
$\Sigma'$ is homeomorphic to the union of two 3-spheres with a common point.

The connected component of $\C^2\times \R\setminus S$
containing the point $(0,0,0,0,1/2) $ is the Levi-flat solution, the complex
leaves being the sections by the hyperplanes $x_3=x_3^0$, $-1<x_3^0<1$.

The sections by the hyperplanes $x_3=x_3^0$ are diffeomorphic to a 3-sphere for
$0<x_3^0<1$ and to the union of two disjoint 3-spheres for $-1<x_3^0<0$, as can be
shown intersecting $S$ by lines through the origin in the
hyperplane $x_3=x_3^0$;
$\Sigma'$ is homeomorphic to the union of two 3-spheres with a common point.

\subsection{Sphere with one special 1-hyperbolic point (sphere with two
horns)}\label{Sec:3.4}

The example of section \ref{Sec:3.3} shows that the necessary conditions of section 2 can be realised. Moreover, from Proposition \ref{Sec:2.8.7}, the hypothesis on the number of complex points is meaningful.

\subsubsection{}\label{Sec:3.4.1}
\begin{prop}\label{Prop:3.4.1}[cf \cite{D08}[Proposition 2.6.1]] Let
$S\subset \C^n$ be a compact connected real 2-codimensional manifold such that
the following holds:

$(i)$ $S$ is a topological sphere; $S$ is nonminimal at every CR point;

$(ii)$ every complex point of $S$ is flat; there exist three special elliptic
points $e_j, j=1,2,3$
and one special $1$-hyperbolic point $h$;

$(iii)$ $S$ does not contain complex manifolds of dimension
$(n-2)$;

$(iv)$ the singular {\rm CR} orbit $\Sigma'$ through $h$ on $S$ is compact and
$\Sigma'\setminus \{h\}$ has two connected components $\sigma_1$ and
$\sigma_2$ whose closures are homeomorphic to spheres of dimension $2n-3$;

$(v)$ the closures $S_1, S_2, S_3$ of the three connected components $S'_1,
S'_2, S'_3$ of
$S\setminus\Sigma'$ are submanifolds with (singular) boundary.

Then each $S_j\setminus \{e_j\cup\Sigma'\}$, $j=1,2,3$
carries a foliation ${\mathcal F}_j$ of class
$C^\infty$ with $1$-codimensional {\rm CR} orbits as compact leaves.
\end{prop}

\begin{proof}
From conditions (i) and (ii), $S$ satisfying
the hypotheses of Proposition \ref{Prop:2.3.1}, near any elliptic
flat point $e_j$, and of Proposition \ref{Prop:2.7.2} near $\Sigma'$, all CR orbits being
diffeomorphic to the sphere ${\bf S}^{2n-3}$.
The assumption (iii) guarantees that
all CR orbits in $S$ must be of real dimension $2n-3$.
Hence, by removing small connected open saturated neighborhoods
of all special elliptic points, and of $\Sigma'$,
we obtain, from $S\setminus\Sigma'$, three compact manifolds $S_j"$,
$j=1,2,3$,
with boundary and
with the foliation ${\mathcal F}_j$ of codimension $1$ given by its CR
orbits whose first
cohomology group with values in ${\bf R}$ is 0, near $e_j$.
It is easy to show that this foliation is transversely oriented.
\end{proof}

\subsubsection{}\label{Sec:3.4.2}
Recall the Thurston's Stability Theorem
($\lbrack$ CaC$\rbrack$, Theorem 6.2.1).
\begin{prop}\label{Prop:3.4.2} Let $(M,{\mathcal F})$ be a compact, connected, transversely-orientable,
foliated manifold with boundary or corners, of codimension 1, of class
$C^1$. If there is a compact leaf $L$ with $H^1(L,{\bf R})=0$, then every leaf
is homeomorphic to $L$ and $M$ is homeomorphic to $L\times \lbrack
0,1\rbrack$, foliated as a product,
\end{prop}

Then, from the above theorem, $S_j"$ is homeomorphic to
${\bf S}^{2n-3}\times [0,1]$ with CR orbits being of the form
${\bf S}^{2n-3}\times\{x\}$ for $x\in [0,1]$. Then the full manifold $S_j$ is
homeomorphic to a half-sphere supported by ${\bf S}^{2n-2}$ and
${\mathcal F}_j$ extends to $S_j$; $S_3$ having its boundary pinched at the point $h$.

\hskip 5cm $\square$\smallskip

\subsubsection{}\label{Sec:3.4.3}
\begin{thm}\label{Thm:3.4.3} Let $S\subset \C^n$, $n\geq 3$, be a
compact connected smooth real $2$-codimensional submanifold
satisfying the conditions $(i)$ to $(v)$ of Proposition \ref{Prop:3.4.2}.
Then there exists a Levi-flat $(2n-1)$-subvariety $\tilde M\subset\C\times\C^n$
with boundary $\tilde S$
(in the sense of currents)
such that the natural projection $\pi: \C\times\C^n\to \C^n$
restricts to a bijection which is a CR diffeomorphism between $\tilde S$ and
$S$ outside the complex points of $S$.
\end{thm}

\begin{proof}
By Proposition \ref{Prop:2.3.1}, for every $e_j$, a
continuous function
$\nu'_j$, $C^{\infty}$ outside $e_j$, can be constructed in a
neighborhood $U_j$ of $e_j$, $j=1,2,3$, and by Proposition \ref{Prop:2.7.2}, we have an
analogous result in a neighborhood of $\Sigma'$.
Furthermore, from Proposition \ref{Prop:3.4.2}, a smooth function $\nu"_j$ whose level
sets are the leaves of ${\mathcal F}_j$ can be obtained
globally on
$S'_j\setminus \{e_j\cup\Sigma'\}$. With the functions $\nu'_j$ and $\nu"_j$, and
analogous functions near $\Sigma'$, then using a partition of unity, we obtain a
global smooth function
$\nu_j\colon S_j\to {\bf R}$ without critical points away from the complex points
$e_j$ and from $\Sigma'$.

Let $\sigma_1$, resp. $\sigma_2$ be the two connected, relatively compact
components of $\Sigma\setminus \{h\}$, according to condition $(iv)$;
$\overline\sigma_1$, resp. $\overline\sigma_2$ are the boundary of $S_1$,
resp. $S_2$, and $\overline \sigma_1\cup\overline\sigma_2$ the boundary of
$S_3$. We can assume that the three functions $\nu_j$ are finite valued and get the same values on $\overline\sigma_1$ and $\overline\sigma_2$.
Hence a function $\nu: S\rightarrow{\bf R}$.

The submanifold $S$ being, locally, a boundary of a Levi-flat hypersurface, is
orientable. We now set $\tilde S=N={\rm gr}\,\nu = \{(\nu(z),z) : z\in S\}$.
Let $S_s=\{e_1,e_2, e_3,\overline{\sigma_1\cup\sigma_2}\}$.

$\lambda: S\rightarrow\tilde S\hskip 2mm\big( z\mapsto \nu ((z),z)\big)$ is
bicontinuous; $\lambda\vert_{S\setminus S_s}$ is a diffeomorphism; moreover
$\lambda$ is a CR map. Choose an orientation on $S$. Then $N$ is an (oriented)
CR subvariety with the negligible set of singularities $\tau=\lambda(S_s)$.

At every point of $S\setminus S_s$, $d_{x_1} \nu\not = 0$, then
condition (H) (section \ref{Sec:3.1.1}) is satisfied at every point of
$N\setminus\tau$.

Then all the assumptions of Theorem \ref{Thm:3.1.2} being satisfied by
$N=\tilde S$, in a particular case, we conclude that $N$ is the
boundary of a Levi-flat
$(2n-2)$-variety (with negligible singularities) $\tilde M$ in ${\bf
R}\times\C^n$.

\smallskip

Taking $\pi: \C\times \C^n\to \C^n$ to be the standard projection,
we obtain the conclusion.
\end{proof}

\subsection{Generalizations: elementary models and their gluing}\label{Sec:3.5}

\subsubsection{}\label{Sec:3.5.1}
The examples and the proofs of the theorems
when $S$ is homeomorphic to a sphere (sections \ref{Sec:3.4})
suggest the following definitions.

\subsubsection{Definitions}\label{Sec:3.5.2}

Let $T'$ be a smooth, locally closed (i.e. closed in an open set), connected submanifold of $\C^n$, $n \geq 3$. We assume
that
$T'$ has the following properties:

$(i)$ $T'$ is relatively compact, non necessarily compact, and of codimension 2.

$(ii)$ $T'$ is nonminimal at every CR point.

$(iii)$ $T'$ does not contain complex manifold of dimension $(n-2)$.

$(iv)$ $T'$ has exactly 2 complex points which are flat and either special elliptic or special 1-hyperbolic.

$(v)$ If $p\in T'$ is special 1-hyperbolic, the singular orbit $\Sigma'$
through
$p$ is compact, $\Sigma'\setminus p$ has two connected components
$\sigma_1$,
$\sigma_2$, whose closures are homeomorphic to spheres of dimension $2n-3$.

$(vi)$ If $p\in T'$ is special 1-hyperbolic, in the neighborhood of
$p$, with convenient coordinates, the equation of $T'$, up to third order terms
is
$$z_n=\sum_{j=1}^{n-1} (z_j\overline z_j+\lambda_j {\mathcal R}e \ z_j^2); \
\lambda_1>1; \ 0\leq\lambda_j<1 \ \ {\rm for} \ \ j\not =1 $$

or in real coordinates $x_j,y_j$
with $z_j=x_j+iy_j$, \quad
$$x_n=\big((\lambda_1+1)x_1^2-(\lambda_1-1)y_1^2\big)+\sum_{j=2}^{n-1}
\big((1+\lambda_j)x_j^2+(1-\lambda_j)y_j^2\big)+ O(|z|^3)$$

\smallskip

$(vii)$ the closures, in $T'$, $T_1, T_2, T_3$ of the three connected components $T'_1, T'_2, T'_3$ of $T'\setminus\Sigma'$ are submanifolds with (singular) boundary. Let $T"_j$, $j=1,2,3$ be neighborhoods of the $T'_j$ in $T'$.

\smallskip

\noindent {\it up- and down-} 1{\it -hyperbolic points}. Let $\tau$ be
the $(2n-2)$-submanifold with (singular) boundary contained into $T'$
such that either $\overline\sigma_1$ (resp. $\overline\sigma_2$) is the
boundary of $\tau$ near $p$, or $\Sigma'$ is the boundary of $\tau$ near $p$.
In the first case, we say that $p$ is 1-{\it up}, (resp. 2-{\it up}), in
the second that $p$ is {\it down}. If $T'$ is contained in a small enough neighborhood of $\Sigma'$ in $\C^n$, such a $T'$ will be called a {\it local
elementary model}, more precisely it defines a {\it germ of elementary model around $\Sigma$}.\smallskip

The union $T$ of $T_1$, $T_2$, $T_3$ and of the germ of elementary model around the singular orbit at every special 1-hyperbolic point is called an {\it elementary model}. $T$ behaves as a locally closed submanifold still denoted $T$.

\subsubsection{Examples of elementary models}\label{Sec:3.5.3}

We will say that $T$ is a {\it elementary model of type}:

$(a)$ if it has: two elliptic points;

$(b)$ if it has: one special elliptic point and one {\it down}-$\{1\}$-hyperbolic point;

$(c_1)$ if it has: one special elliptic point and one 1-{\it up}-$\{1\}$-hyperbolic point;

$(c_2)$ if it has: one special elliptic point and one 2-{\it up}-$\{1\}$-hyperbolic point;

$(d_1)$ if it has: two special 1-{\it up}-$\{1\}$-hyperbolic points;

$(d_2)$ if it has: two special 2-{\it up}-$\{1\}$-hyperbolic points;

$(e)$ if it has: two special {\it down}-$\{1\}$-hyperbolic points;

\smallskip

Other configurations are easily imagined.

The prescribed boundary of a Levi-flat hypersurface of $\C^n$ in \cite{DTZ05} and \cite{DTZ10}, whose complex points are flat and elliptic, is an elementary model of type $(a)$.

\subsubsection{Properties of elementary models}\label{Sec:3.5.4}

For instance, $T$ is 1-up and has one special elliptic point, we solve the boundary problem as in $S_1$ in the proof of Theorem \ref{Thm:3.4.3}.
\begin{prop}\label{Prop:3.5.4'} Let $T$ be a local elementary model.
Then, $T$ carries a foliation ${\mathcal F}$ of class
$C^\infty$ with $1$-codimensional {\rm CR} orbits as compact leaves.
\end{prop}

\begin{proof} From the definition at the end of section \ref{Sec:3.5.2} and Proposition \ref{Prop:2.7.2}.\end{proof}

\subsubsection{}\label{Sec:3.5.5}
\begin{thm}\label{Thm:3.5.5} Let $T$ be the elementary model there exists an open neighborhood $T"$ in $T'$ carrying a smooth function $\nu: T"\rightarrow \R$ whose level sets are the leaves of a smooth foliation.
\end{thm}
\begin{proof}
By removing small connected open saturated neighborhoods
of every special elliptic point, and of $\Sigma'$, the singular orbit through every special 1-hyperbolic point $p$,
we obtain, from $S\setminus\Sigma'$, three compact manifolds $S_j"$,
$j=1,2,3$,
with boundary,

$(a)$ $S_1$ and $S_2$ containing one special elliptic point $e$ or one special 1-hyperbolic point with the foliations ${\mathcal F}_1$, ${\mathcal F}_2$, from Propositions \ref{Prop:2.3.1} and \ref{Prop:3.5.4'},

$(b)$ $S_3"$
with the foliation ${\mathcal F}_3$ of codimension $1$ given by its CR
orbits whose first
cohomology group with values in ${\bf R}$ is 0, near $e$, or $p$.
It is easy to show that this later foliation is transversely oriented.

\smallskip

From the Thurston's Stability Theorem (see section \ref{Sec:3.4.2}), $S_3"$ is homeomorphic to
${\bf S}^{2n-3}\times [0,1]$, foliated as a product, with CR orbits being of the form
${\bf S}^{2n-3}\times\{x\}$ for $x\in [0,1]$; hence smooth functions $\nu_1$, $\nu_2$, $\nu_3$, whose level sets are the leaves of the foliations ${\mathcal F}_1$, ${\mathcal F}_2, {\mathcal F}_3$ respectively, and using a partition of unity the desired function $\nu$ on $T$.

\end{proof}

\subsection{}\label{Sec:3.6}
\begin{thm}\label{Thm:3.6} Let $T$ be an elementary model. Then there exists a Levi-flat $(2n-1)$-subvariety $\tilde M\subset \C\times\C^n$
with boundary $\tilde T$
(in the sense of currents)
such that the natural projection $\pi: \C\times\C^n\rightarrow\C^n$
restricts to a bijection which is a CR diffeomorphism between $\tilde T$ and
$T$ outside the complex points of $T$.
\end{thm}
\begin{proof} The submanifold $T$ being, locally, a boundary of a Levi-flat hypersurface, is
orientable. We now set $\tilde T=N={\rm gr}\,\nu = \{(\nu(z),z) : z\in S\}\subset E\cong{\bf R}\times\C^{n-1}$.
Let $T_s$ be the union of the flat complex points of $T$.

$\lambda: T\rightarrow\tilde T\hskip 2mm\big( z\mapsto \nu ((z),z)\big)$ is
bicontinuous; $\lambda\vert_{T\setminus T_s}$ is a diffeomorphism; moreover
$\lambda$ is a CR map. Choose an orientation on $T$. Then $N$ is an (oriented)
CR subvariety with the negligible set of singularities $\tau=\lambda(T_s)$.

Using Remark \ref{Rem:3.1.3}, at every point of $T\setminus T_s$, $d_{x_1} \nu\not = 0$, we see that
condition (H) (section \ref{Sec:3.1.1}) is satisfied at every point of
$N\setminus\tau$.

Then all the assumptions of Theorem \ref{Thm:3.1.2} being satisfied by
$N=\tilde T$, in a particular case, we conclude that $N$ is the
boundary of a Levi-flat
$(2n-2)$-variety (with negligible singularities) $\tilde M$ in ${\bf
R}\times\C^n$.

\smallskip

Taking $\pi: \C\times \C^n\to \C^n$ to be the standard projection,
we obtain the conclusion.
\end{proof}

\subsection{Gluing of elementary models}\label{Sec:3.7}

\subsubsection{}\label{Sec:3.7.1}
The gluing happens between two compatible
elementary models along boundaries, for instance down and 1-up. Remark that the gluing can only be made at special 1-hyperbolic points. More precisely, it can be defined as follows.

\smallskip

The assumed properties of the submanifold $S$ in section 2 in $\C^n$ have a meaning in any complex analytic manifold $X$ of complex dimension $n\geq 3$, and are kept under any holomorphic isomorphism.

We will define a submanifold $S'$ of $X$ obtained by gluying of elementary models by induction on the number $m$ of models. An elementary model $T$ in $X$ is the image of an elementary model $T_0$ in $\C^n$ by an analytic isomorphism of a neighborhood of $T_0$ in $\C^n$ into $X$.

\subsubsection{}\label{Sec:3.7.2}
Let $S'$ be a closed smooth real submanifold of $X$ of dimension $2n-2$ which is non minimal at every CR point. Assume that $S'$ is obtained by gluing of $m$ elementary models.

a) $S'$ has a finite number of flat complex points, some special elliptic and the others special 1-hyperbolic;

b) for every special 1-hyperbolic $p'$, there exists a CR-isomorphism $h$ induced by a holomorphic isomorphism of the ambient space $\C^n$ from a neighborhood of $p$ in $T'$ onto a neighborhood of $p'$ in $S'$.

c) for every CR-orbit $\Sigma_{p'}$ whose closure contains a special 1-hyperbolic point $p'$, there exists a CR-isomorphism $h$ induced by a holomorphic isomorphism of the ambient space $\C^n$ from a neighborhood of $\Sigma_p=\Sigma_p'\setminus p$ in $T'$ onto a neighborhood $V$ of $\Sigma_{p'}$ in $S'$.

Every special 1-hyperbolic point of $S'$ which belongs to only one elementary model in $S'$ will be called {\it free}.

We will define the gluing of one more elementary model to $S'$.

\subsubsection{Gluing an elementary model $T$ of type $(d_1)$ to a free down-1-hyperbolic point of $S'$}\label{Sec:3.7.3}
Let $h_1$ be a CR-isomorphism from a neighborhood $V_1$ of $\overline\sigma'_1$ induced by a holomorphic isomorphism of the ambient space $\C^n$ onto a neighborhood of $\sigma_1$ in $S'$. Let $k_1$ be a CR-isomorphism from a neighborhood $T"_1$ of $T'_1$ into $X$ such that $k_1\vert V_1=h_1$.

\subsubsection{}\label{Sec:3.7.4}
\begin{thm}\label{Thm:3.7.4}
The compact manifold or the manifold with singular boundary $S'$, obtained by the gluing of a finite number of elementary models, is the boundary of a Levi-flat hypersurface of $X$ in the sense of currents.
\end{thm}

\begin{proof}From Theorem \ref{Thm:3.6} and the definition of gluing.\end{proof}

\subsection{Examples of gluing}\label{Sec:3.8}

Denoting the gluing of the two models of type $(d_1)$ and $(d_2)$ to a free down-1-hyperbolic point of $S'$ by: $\rightarrow (d_1)-(d_2)$, and the converse by: $(d_1)-(d_2)\rightarrow$, and, also, analogous configurations in the same way, we get:

torus: $(b)\rightarrow (d_1)-(d_2)\rightarrow (b)$; the Euler-Poincar\'e characteristic of a torus is $\chi({\bf T}^k)=0$: 2 special elliptic and 2 special 1-hyperbolic points.

bitorus: $(b)\rightarrow (d_1)-(d_2)\rightarrow (e)\rightarrow (d_1)-(d_2)\rightarrow (b)$.

\section{Case of graphs}\label{Sec:4}

(see \cite{DTZ09} for the case of elliptic points only, and dropping the property of the function solution to be Lipschitz).

\subsection{}\label{Sec:4.1}
We want to add the following hypothesis: $S$ is embedded into the boundary of a strictly pseudoconvex domain of $\C^n$, $n\geq 3$ , and more precisely, let $(z,w)$ be the coordinates in $\C^{n-1}\times \C$, with $z=(z_1, \ldots,z_{n-1}), w=u+iv=z_n$, let $\Omega$ be a strictly pseudoconvex domain of $\C^{n-1}\times\R_u$ (i.e. the second fundamental form of the boundary $b\Omega$ of $\Omega$ is everywhere positive definite); let $S$ be the graph $gr (g)$ of a smooth function $g: b\Omega\rightarrow \R_v$. notice that $b\Omega\times\R_v$ contains $S$ and is strictly pseudoconvex.

Assume that $S$ is a {\it horned sphere} (section \ref{Sec:3.4}), {\it satisfying the hypotheses of} Theorem \ref{Thm:3.4.3}. Denote by $p_j$, $j=i,\ldots,4$ the complex points of $S$. Our aim is to prove

\subsection{}\label{Sec:4.2}
\begin{thm}\label{Thm:4.2}
Let $S$ be the graph of a smooth function $g: b\Omega\rightarrow \R_v$. Let $Q=(q_1,\ldots, q_4)\in b\Omega$ be the projections of the complex points $P=(p_1,\ldots,p_4)$ of $S$, respectively. Then, there exists a continuous function $f:\overline\Omega\rightarrow \R_v$ which is smooth on $\overline\Omega\setminus Q$ and such that $f_{\vert b\Omega}=g$, and $M_0=graph(f)\setminus S$ is a smooth Levi flat hypersurface of $\C^n$. Moreover, each complex leaf of $M_0$ is the graph of a holomorphic function $\phi:\Omega'\to \C$ where $\Omega'\subset\C^{n-1}$
is a domain with smooth boundary (that depends on the leaf) and $\phi$ is smooth on $\overline\Omega'$.
\end{thm}

The natural candidate to be the graph $M$ of $f$ is $\pi\big(\tilde M\big)$ where $\tilde M$ and $\pi$ are as in Theorem \ref{Thm:3.4.3}. We prove that this is the case proceeding in several steps.

\subsection{Behaviour near $S$}\label{Sec:4.3}

\subsubsection{}\label{Sec:4.3.1}

{\it Assume that $D$ is a strictly pseudoconvex domain and that $S\subset bD$.}

Recall (\cite{HL75}[Theorem 10.4]: {\it Let $D$ be a strictly pseudoconvex domain of $\C^n$, $n\geq 3$ with boundary $bD$, $\Sigma\subset bD$ be a compact connected maximally complex smooth $(2d-1)$-submanifold with $d\geq 2$. Then, $\Sigma$ is the boundary of a uniquely determined relatively compact subset $V\subset \overline D$ such that $\overline V\setminus \Sigma$ is a complex analytic subset of $D$ with finitely many singularities of pure dimension $\leq d-1$, and near $\Sigma$, $\overline V$ is a $d$-dimensional complex manifold with boundary.}

$V$ is said to be {\it the solution of the boundary problem for} $\Sigma$.

\subsubsection{}\label{Sec:4.3.2}
\begin{lem}[\cite{DTZ09}]\label{Lem:4.3.2} Let $\Sigma_1$, $\Sigma_2$ be compact connected maximally complex $(2d-1)$-submanifolds of $bD$. Let $V_1$, $V_2$ be the corresponding solutions of the boundary problem. If $d\geq 2$, $2d\geq n+1$ and $\Sigma_1\cap\Sigma_2=\emptyset$, then $V_1\cap V_2=\emptyset$.
\end{lem}

Let $\Sigma$ be a CR orbit of the foliation of $S\setminus P$. Then $\Sigma$ is a compact maximally complex $(2n-3)$-dimensional real submanifold of $\C^n$ contained in $bD$. Let $V=V_\Sigma$ be the solution of the boundary problem corresponding to $\Sigma$. From Theorem \ref{Thm:3.4.3}, $V=\pi(\tilde V)$, where $\tilde V=(\tilde M\setminus\tilde S)\cap (\C^n\times \{x\})$ for suitable $x\in (0,1)$, the projection on the $x$-axis being finite, we can always assume that it lies into $(0,1)$. Moreover $\pi_{\vert\tilde V}$ is a biholomorphism $\tilde V \cong V$ and $M\setminus S\subset D$.

Let $\Sigma_1$, $\Sigma_2$ be two distinct orbits of the foliation of $S\setminus P$, and $\overline V_1, \overline V_2$ the corresponding leaves, then, from Lemma \ref{Lem:4.3.2}, $\overline V_1\cap \overline V_2=\emptyset$.

\subsubsection{}\label{Sec:4.3.3}

{\it Assume that $S$ satisfies the full hypotheses of} Theorem \ref{Thm:4.2}.

\smallskip

Set $m_1=\min\limits_S g$, $m_2=\max\limits_S g$ and $r\gg 0$ such that
$$
D=\Omega\times[m_1,m_2]\subset\subset{
\bf B(r)}\cap\big(\Omega\times i\R_v\big)
$$
where ${\bf B(r)}$ is the ball $\{\vert(z,w)\vert<r\}$.

\subsubsection{}\label{Sec:4.3.4}
\begin{lem}\label{Lem:4.3.4}  Let $p\in S$ be a CR point. Then, near $p$, $M$ is the graph of a function $\phi$ on a domain $U\subset\C_z^{n-1}\times\R_u$ which is smooth up to the boundary of $U$.\end{lem}

\begin{proof} Near $p$, each CR orbit $\Sigma$ is smooth and can be represented as the graph of a CR function over a strictly pseudoconvex hypersurface and $V_\Sigma$ as the graph of the local holomorphic extension of this function. From Hopf lemma, $V$ is transversal to the strictly pseudoconvex hypersurface $d\Omega\times i\R_v$ near $p$. Hence the family of the $V_\Sigma$, near $p$, forms a smooth real hypersurface with boundary on $S$ that is the graph of a smooth function $\phi$ from a relative open neighborhood $U$ of $p$ on $\overline \Omega$ into $\R_v$. Finally, Lemma \ref{Lem:4.3.2} garantees that this family does not intersect any other leaf $V$ from $M$.
\end{proof}

\subsubsection{}\label{Sec:4.3.5}
\begin{cor}\label{Cor:4.3.5} If $p\in S$ is a CR point, each complex leaf $V$ of $M$, near $p$, is the graph of a holomorphic function on a domain $\Omega_V\subset \C_z^{n-1}$, which is smooth up to the boundary of $\Omega_V$.
\end{cor}

\subsection{Solution as a graph of a continuous function}\label{Sec:4.4}

\subsubsection{}\label{Sec:4.4.1}
Recall results of Shcherbina \cite{Sh93} from:

(a) the Main Theorem:

{\it Let $G$ be a bounded strictly convex domain in $\C_z\times\R_u$ ($z\in\C$) and $\varphi:bG\rightarrow \R_v$ be a continuous function. Then the following properties hold, where $\Gamma= gr$, and $\hat\Gamma(\varphi)$ means polynomial hull of $\Gamma(\varphi)$:

$(a_i)$ the set $\hat\Gamma(\varphi)\setminus\Gamma(\varphi)$ is the union of a disjoint family of complex discs} $\{D_\alpha\}$;

$(a_{ii})$ {\it for each $\alpha$, there is a simply connected domain $\Omega_\alpha\subset\C_z$ and a holomorphic function $w=f_\alpha$, defined on $\Omega_\alpha$, such that $D_\alpha$ is the graph of $f_\alpha$.}

$(a_{iii})$ {\it For each $f_\alpha$, there exists an extension $f_\alpha^*\in C(\overline\Omega_\alpha)$ and $bD_\alpha=\{(z,w)\in b\Omega_\alpha\times \C_w: w=f_\alpha^*(z)\}$}.

(b)
\begin{lem}\label{Lem:2.4}
Let $\{G_n\}_{n=0}^\infty$, $G_n\subset\C_z\times\R_u$, be a sequence of bounded strictly convex domains such that $G_n\rightarrow G_0$. Let $\{\varphi_n\}_{n=0}^\infty$, $\varphi_n:\partial G_n\rightarrow\R_v$ be a sequence of continuous functions such that $\Gamma(\varphi_n)\rightarrow \Gamma(\varphi_0)$ in the Hausdorff metric. Then, if $\Phi_n$ is the continuous function $:\overline G_n\rightarrow\R_v$ such that $\hat\Gamma(\varphi)=\Gamma(\Phi)$, we have $\Gamma(\Phi_n)\rightarrow \Gamma(\Phi_0)$ in the Hausdorff metric.\end{lem}

(c)
\begin{lem}\label{Lem:3.2} Let ${\mathcal U}$ be a smooth connected surface which is properly embedded into some convex domain $G\subset \C_z\times\R_u$. Suppose that near each point of this surface, it can be defined locally by the equation $u=u(z)$. Then the surface ${\mathcal U}$ can be represented globally as a graph of some function $u=U(z)$, defined on some domain $\Omega\subset\C_z$.\end{lem} 

\subsubsection{}\label{Sec:4.4.2}
\begin{prop}\label{Prop:4.4.2} $M$ is the graph of a continuous function $f:\overline\Omega\rightarrow \R_v$.\end{prop}
\begin{proof} We will intersect the graph $S$ with a convenient affine subspace of real dimension 4 to go back to the situation of Shcherbina.

Fix $a\in (\C_z^{n-1}\setminus 0)$ and, for a given point $(\zeta,\xi)\in\Omega$, with
$\zeta\in\C_z^{n-1}$ and $\xi\in\R_u$, let $H_{(\zeta,\xi)}\subset \C_z^{n-1}\times\{\xi\}$ be the complex line through $(\zeta,\xi)$ in the direction $(a,0)$. Set:

$$L_{(\zeta,\xi)}=H_{(\zeta,\xi)}+\R_u(0,1),\hskip 3mm \Omega_{(\zeta,\xi)}=L_{(\zeta,\xi)}\cap\Omega,\hskip 3mm S_{(\zeta,\xi)}=(H_{(\zeta,\xi)}+\C_w(0,1))\cap S$$

Then $S_{(\zeta,\xi)}$ is contained in the strictly convex cylinder
$$(H_{(\zeta,\xi)}+\C_w(0,1))\cap (b\Omega\times i\R_v)$$
and is the graph of $g_{\vert b\Omega_{(\zeta,\xi)}}$.

From $(a_{ii})$, the polynomial hull of $S_{(\zeta,\xi)}$ is a continuous graph over $\overline\Omega_{(\zeta,\xi)}$. Consider $M=\pi(\tilde M)$ and set

$$M_{\zeta, \xi)}=(H_{(\zeta,\xi)}+\C_w (0,1))\cap M.$$

It follows that $M_{\zeta, \xi)}$ is contained in the polynomial hull $\hat S_{(\zeta,\xi)}$. From $(a_{iii})$, $\hat S_{(\zeta,\xi)}$ is a graph over $\overline \Omega_{(\zeta,\xi)}$ foliated by analytic discs, so $M_{\zeta, \xi)}$ is a graph over a subset $U$ of $\overline \Omega_{(\zeta,\xi)}$.

Every analytic disc $\Delta$ of $\hat S_{(\zeta,\xi)}$ had its boundary on $S_{(\zeta,\xi)}$. Since all the the complex points of $S$ are isolated, $b\Delta$ contains a CR point $p$ of $S$; from Lemma \ref{Lem:4.3.4}, near $p$, $M_{\zeta, \xi)}$ is a graph over $\overline\Omega_{(\zeta,\xi)}$. Near $p$, $\Delta$ is contained in $M_{\zeta, \xi)}$, then in a closed complex analytic leaf $V_\Sigma$ of $M$; so $\Delta\subset V_\Sigma\subset M$; but $\Delta\subset H_{(\zeta,\xi)}+\C_w (0,1)$; then: $\Delta\subset M_{\zeta, \xi)}$.
Consequently, near $p$, $M_{\zeta, \xi)}=\hat S_{(\zeta,\xi)}$.

It follows that $M$ is the graph of a function $f:\overline\Omega\rightarrow\R_v$.

One proves, using (b), that $f$ is continuous on $\Omega$, whence on $\overline\Omega\setminus Q$, by Lemma \ref{Lem:4.3.4}. Then continuity at every $q_j$ is proved using the Kontinuit\"atsatz on the domain of holomorphy $\Omega\times i\R_v$.
\end{proof}

\subsection{Regularity}\label{Sec:4.5}

The property: {\it $M\setminus P=(p_1,\ldots,p_4)$ is a smooth manifold with boundary} results from:

\subsubsection{}\label{Sec:4.5.1}
\begin{lem}\label{Lem:4.5.1}
 Let $U$ be a domain of $\C_z^{n-i}\times\R_u, n\geq 2, f:U\rightarrow \R_v$ a continuous function. Let $A\subset {\rm graph}(f)$ be a germ of complex analytic set of codimension 1. Then $A$ is a germ of complex manifold which is a graph of over $\C_z^{n-i}$.
\end{lem}
\begin{proof}
Assume that $A$ is a germ at 0. Let $g\in {\mathcal O} , h\not= 0$ such that $A=\{h=0\}$. For $\varepsilon<<1$, let ${\bf D}_\varepsilon$ be the disc $\{z=0\}\cap\{\vert w\vert<\varepsilon\}$
, then $A\cap {\bf D}_\varepsilon=\{0\}$, i.e. $A$ is $w$-regular.

Let $\pi:\C^n_{z,w}\rightarrow\C^{n-1}_z$ be the projection. The local structure theorem for analytic sets gives:

{\it for some neighborhood $U$ of 0 in $\C^{n-1}_z$, there exists an analytic hypersurface $\Delta\subset U$ such that: $A_\Delta=A\cup((U\setminus\Delta)\times{\bf D}_\varepsilon)$ is a manifold;

$\pi/A_\Delta\rightarrow U\setminus\Delta$ is a $d (\in \N)$-sheeted covering.}

It is easy to show that the covering $\pi:A_\Delta\rightarrow U\setminus \Delta$ is trivial.

Then we may define $d$ holomorphic functioons $\tau_1,\ldots,\tau_d: U\setminus\Delta\rightarrow\C$ such that $A_\Delta$ is the union of the graphs of the $\tau_j$. By the Riemann extension theorem, the functions $\tau_j$ extend as holomorphic functions $\tau_j\in {\mathcal O}(U)$. Suppose that $\tau_j\not = \tau_k$, for $j\not = k$, then for some disc ${\bf D}\subset U$ centered at 0, we have $\tau_j\vert {\bf D}\not = \tau_k\vert {\bf D}$, then $(\tau_j-\tau_k)\vert_{\bf D}$ vanishes only at 0. But, from the hypothesis, in restriction to ${\bf D}$, $\{Re(\tau_j-\tau_k)=0\} \subset\{\tau_j-\tau_k=0\}\vert_{\bf D}=\{0\}$, impossible. \end{proof}

\subsection{}\label{Sec:4.6} \begin{proof}[Proof of the Theorem \ref{Thm:4.2}]

Consider the foliation of $S\setminus P$ given by the level sets of the smooth function $\nu: S\rightarrow\lbrack 0,1\rbrack$ (sections \ref{Sec:2.3} and \ref{Sec:2.7}) and set $L_t=\{\nu=t\}$ for $t\in (0.1)$. Let $V_t\subset \overline\Omega\times i\R_v\subset \C^n$ be the complex leaf of $M$ bounded by $L_t$.

By Proposition \ref{Prop:4.4.2}, $M$ is the graph of a continuous function over $\Omega$, and, by Lemma \ref{Lem:4.5.1}, each leaf $V_t$ is a complex smooth hypersurface and $\pi\vert_{V_t}$ is a submersion.

$\rightarrow$

Since $\Omega$ is strictly convex, as in Shcherbina (see 4.4.1, c)), $\pi_{\vert V_t}$ is 1-1, then, by Corollary \ref{Cor:4.3.5}, $\pi$ sends $V_t$ onto a domain $\Omega_t\subset\C_z^{n-1}$ with smooth boundary. Let

$\pi_u:(\C_z^{n-1}\times\R_u)\times i\R_v\rightarrow\R_u$

$\pi_v:(\C_z^{n-1}\times\R_u)\times i\R_v\rightarrow\R_v$

then $\pi_{u\vert L_t}=a_t.\pi_{\vert L_t}$ and $\pi_{v\vert L_t}=b_t.\pi_{\vert L_t}$ where $a_t$, $b_t$ are smooth functions on $b\Omega_t$. Moreover $b\Omega_t$, $a_t$, $b_t$ depend smoothly on $t$.

If $(z_t,w_t)\in M$, then $w_t$ varies on $V_t$, so $w_t$ is the holomorphic extension of $a_t+ib_t$ to $\Omega_t$. In particular $u_t$ and $v_t$ are smooth in $(z,t)$, from the Bochner-Martinelli formula.

$\displaystyle \frac{\partial u_t}{\partial t}$ is harmonic on $\Omega_t$ for each $t$ and has a smooth extension on $b\Omega_t$.

From Lemma \ref{Lem:4.3.4} and Corollary \ref{Cor:4.3.5}, $\displaystyle \frac{\partial u_t}{\partial t}$ does not vanish on $b\Omega_t$. Since the CR orbits $L_t$ are connected from Proposition \ref{Prop:3.4.1}, $b\Omega_t$ is also connected, hence $\displaystyle \frac{\partial u_t}{\partial t}$ has constant sign on $b\Omega_t$. Then, by the maximum principle, also $\displaystyle \frac{\partial u_t}{\partial t}$ on $\Omega_t$ and, in particular does not vanish. This implies that $M\setminus S$ is the graph of a smooth function over $\Omega$ which smoothly extends to $\overline\Omega\setminus Q$.

From Proposition \ref{Prop:4.4.2}, $M$ is the graph of a continuous
function over $\overline\Omega$.
\end{proof}

\subsection{Elementary smooth models}\label{Sec:4.7}

\subsubsection{Definition}\label{Sec:4.7.1}

{\it An elementary smooth model in $\C^n$} is an elementary model in the sense of section \ref{Sec:3.5.2} and satisfying the further condition which makes sense from Theorem \ref{Thm:4.2}:

(G) Let $(z,w)$ be the coordinates in $\C^{n-1}\times \C$, with $z=(z_1, \ldots,z_{n-1}), w=u+iv=z_n$, let $\Omega$ be a strictly pseudoconvex domain of $\C^{n-1}\times\R_u$; assume that $T'$ is the graph of a smooth function $g: b\Omega\rightarrow \R_v$.

\subsubsection{}\label{Sec:4.7.2}
\begin{thm}\label{Thm:4.7.2}
Let $T$ be an elementary smooth model. Then, there exists a continuous function $f:\overline\Omega\rightarrow \R_v$ which is smooth on $\overline\Omega\setminus Q$ and such that $f_{\vert b\Omega}=g$, and $M_0=graph(f)\setminus S$ is a smooth Levi flat hypersurface of $\C^n$; in particular, $S$ is the boundary of the hypersurface $M=graph(f)$
\end{thm}

\begin{proof} similar to the proof of Theorem \ref{Thm:4.2}.\end{proof}

\subsubsection{Gluing of elementary smooth models}\label{Sec:4.7.3}

In an open set of $\C^n$, a coordinate system $(z,w)$ of $\C^{n-1}_z\times\R_u$ defines an $(n-1,1)$-{\it frame}.

To define the gluing of elementary models (section \ref{Sec:3.7}) we considered a CR-isomorphism from an open set of $\C^n$ induced by a holomorphic isomorphism of the ambient space $\C^n$ onto a an open set of $\C^n$. To define the gluing of elementary smooth models, we have to consider a holomorphic isomorphism of the ambient space $\C^n$ onto an open set of $\C^n$ sending an $(n-1,1)$-frame of $\C^{n-1}_z\times\R_u$ onto an $(n-1,1)$-frame of $\C^{n-1}_{z'}\times\R_{u'}$.

As in section \ref{Sec:3.7.1}, we will define a submanifold $S'$ of $X$ obtained by gluing of elementary smooth models by induction on the number $m$ of models. An elementary smooth model $T$ in $X$ is
the image of an elementary smooth model $T_0$ of $\C^n$ by an analytic isomorphism of a neighborhood of $T_0$ in $\C^n$ into $X$.

\smallskip

\noindent {\it Gluing an elementary smooth model $T$ of type $(d_1)$ to a free down-1-hyperbolic point of $S'$}.

Every elementary smooth model is contained in a cylinder $b\Omega\times\R_v$ determined by $\Omega$ and an $(n-1,1)$-frame. Two sets $\Omega$ are {\it compatible} if either they coincide or one is part of the other.

The announced gluing is defined in the following way: there exists a CR-isomorphism $h_1$ from a neighborhood $V_1$ of $\overline\sigma'_1$ induced by a holomorphic isomorphism of the ambient space $\C^n$ onto a neighborhood of $\sigma_1$ in $S'$. Let $k_1$ be a CR-isomorphism from a neighborhood $T"_1$ of $T'_1$ into $X$ such that $k_1\vert V_1=h_1$, and there exists a common $(n-1,1)$-frame on which the corresponding sets $\Omega$ are compatible.
The existence of such a situation is possible as the example of the horned (almost everywhere) smooth sphere shows (Theorem \ref{Thm:4.2}.).

Remark that the gluing implies that the obtained submanifold $S'$ is $C^0$ and smooth except at the complex points.

Other gluing are obtained in a similar way. Hence:

\begin{thm}
The manifold $S'$ obtained by gluing of elementary smooth models is of class $C^0$, and smooth except at the complex points.
\end{thm}

\begin{cor}
The manifold $S'$ is the boundary of manifold
$M$ of class $C^\infty$ whose interior is a Levi-flat smooth hypersurface.
\end{cor}

\providecommand{\bysame}{\leavevmode\hbox to3em{\hrulefill}\thinspace}
\providecommand{\MR}{\relax\ifhmode\unskip\space\fi MR }
\providecommand{\MRhref}[2]{%
  \href{http://www.ams.org/mathscinet-getitem?mr=#1}{#2}
}
\providecommand{\href}[2]{#2}

\end{document}